\documentclass[reqno]{amsart}
\usepackage{float}
\usepackage{amsmath}
\usepackage{graphicx}
\usepackage{latexsym}
\usepackage{amsfonts}
\usepackage{amssymb}
\usepackage{color}
\theoremstyle{plain}
\newtheorem{theorem}{Theorem}

\newtheorem{lemma}[theorem]{Lemma}
\newtheorem{proposition}[theorem]{Proposition}
\theoremstyle{definition}

\newtheorem*{remark*}{Remark}

\newcommand{\pr}{\mathbf P}
\newcommand{\e}{\mathbf E}

\begin{document}
\title[Area under excursion]{Tail asymptotics for the area under the excursion of a  random walk with  heavy-tailed increments}

\author[Denisov]{Denis Denisov}
\address{School of Mathematics, University of Manchester, Oxford Road, Manchester M13 9PL, UK}
\email{denis.denisov@manchester.ac.uk}

\author[Perfilev]{Elena Perfilev} 
\address{Institut f\"ur Mathematik, Universit\"at Augsburg, 86135 Augsburg, Germany}
\email{Elena.Perfilev@math.uni-augsburg.de}

\author[Wachtel]{Vitali Wachtel} 
\address{Institut f\"ur Mathematik, Universit\"at Augsburg, 86135 Augsburg, Germany}
\email{vitali.wachtel@mathematik.uni-augsburg.de}

\begin{abstract}
We study tail behaviour of the distribution of the area under the positive excursion of a random walk which has negative drift and heavy-tailed increments. 
We determine the asymptotics for tail probabilities for the area.
\end{abstract}
\maketitle

\section{Introduction and statement of results}

Let $\{S_n; n\geq 1\}$ be a random walk with i.i.d. increments $\lbrace X_k; k\geq 1\rbrace$.
We shall assume that the increments have negative expected value, $\e X_1=-a$. 
Let $\overline F(x)=\pr (X_1>x)$ be the tail distribution function of $X_1$. 
Let  
\[
\tau:=\min\lbrace n\geq1: S_n\leq 0 \rbrace
\]
be the first time the random walk exits the positive half-line. 
We consider the area under the random walks excursion $\{S_1,S_2,\ldots,S_{\tau-1}\}$:
\[
A_\tau:=\sum_{k=0}^{\tau-1}S_k.
\]
Since $\tau$ is finite almost surely, the area $A_\tau$ is finite as well. 
In this note we will study asymptotics for $\pr(A_\tau>x),$ as $x\to \infty$, 
in the case when distribution of increments is  heavy-tailed. 
This paper continues the research of~\cite{PW17}, where the light-tailed case has been considered. 

The heavy-tailed asymptotics for $\pr(A_\tau>x)$ was studied previously 
by Borovkov, Boxma and Palmowski~\cite{BBP03}. 
They considered the case when the increments of the random walk have a distribution with regularly varying tail, that is $\overline F(x)=x^{-\alpha }L(x),$ where $L(x)$ is a slowly varying function. 
For $\alpha>1$ they showed 
\begin{equation}
\label{reg.tail} 
\pr(A_\tau>x)\sim\e\tau\overline F(\sqrt{2ax}),\quad x\to\infty.
\end{equation}
These asymptotics can be explained by a traditional heavy-tailed one big jump heuristics. 
In order to have a huge area, the random walk should have a large jump, say $y$, at the very beginning of the excursion. 
After this jump the random walk goes down along the line $y-an$ according to the Law of Large Numbers. 
Thus, the duration of the excursion should approximately be around $y/a$. 
As a result, the area will be of order $y^2/2a$. 
Now, from the equality $x=y^2/2a$ one infers that a jump of order $\sqrt{2ax}$ is needed. 
Since the same strategy is valid for the maximum $M_\tau:=\max_{n<\tau}S_n$ of the first excursion, one can rewrite \eqref{reg.tail} in the
following way:
\[
\mathbf{P}(A_\tau>x)\sim\mathbf{P}(M_\tau>\sqrt{2ax}),\quad x\to\infty.
\]

However, the class of regularly varying distributions does not include all subexponential distributions and excludes, in particular, log-normal distribution and Weibull distribution with parameter $\beta<1$. 
The asymptotics for these remaining cases have been put as an open problem in~\cite[Conjecture 2.2]{KP11} for a strongly related workload process.  
We will reformulate this conjecture as follows
\begin{equation}\label{eq:conj}
 \pr(A_\tau>x)\sim \pr\left(\tau>\sqrt{\frac{2x}{a}}\right), \quad x\to\infty, 
\end{equation}
when $F\in\mathcal S$ and $\mathcal S$ is a sublclass of subexponential distributions. 
Note that using the asymptotics for 
\begin{equation}\label{eq:asymp.tau}
\pr(\tau>x)\sim \e\tau\overline F(a x) 
\end{equation}
from~\cite{DS13}  for Weibull distributions with parameter $\beta<1/2$, one can see that in this case asymptotics~\eqref{eq:conj} is equivalent to~\eqref{reg.tail}. 
In this note we partially settle \eqref{eq:conj}. It is not difficult to show that the same arguments hold for the workload process and to prove the same asymptotics for the area of the workload process, thus settling the original~\cite[Conjecture 2.2]{KP11}.  
In passing we  note that it is doubtful that~\eqref{eq:conj} holds in full. 
The reason for that is that for both $\tau$ and $A_\tau$ the asymptotics~\eqref{eq:asymp.tau} and~\eqref{eq:conj} are no longer valid for Weibull distributions with parameter $\beta>1/2$. The analysis for $\beta>1/2$ involves more complicated optimisation procedure leading to a Cramer series and it is unlikely that the answers will be the same for the area and for the exit time.    

\subsection{Main results}

We will now present the results. We will start with the regularly varying case. 
In this case the  connection between the tails of $A_\tau$ and $M_\tau$ is strong and we will be able to use the asymptotics for $\pr(M_\tau>x)$ found in~\cite{FPZ}, see also a short proof in~\cite{D2005}, to find the asymptotics for $\pr(A_\tau>x)$. 
\begin{proposition}
\label{prop:joint}
We have the following two statements.
\begin{itemize}
 \item[(a)] If $\overline{F}(x):=\mathbf{P}(X_1>x)=x^{-\alpha}L(x)$ with some $\alpha\ge1$ and
            $\mathbf{E}|X_1|<\infty$ then, uniformly in $y\in[\varepsilon\sqrt{x},\sqrt{2ax}]$,
            \begin{equation}
            \label{joint.1}
            \mathbf{P}(A_\tau>x,M_\tau>y)\sim \mathbf{E}\tau\overline{F}(\sqrt{2ax}).
            \end{equation}

 \item[(b)] If $\overline{F}(x)\sim x^{-\varkappa} e^{-g(x)}$, 
 where $g(x)$ is a monotone continuously differentiable function satisfying 
 $\frac{g(x)}{x^\beta}\downarrow$ for $\beta\in(0,1/2)$, 
            and $\mathbf{E}|X_1|^\varkappa<\infty$ for some 
            $\varkappa>1/(1-\beta)$  then
            \eqref{joint.1}  holds uniformly in
            $y\in\left[\sqrt{2ax}-\frac{R\sqrt{2ax}}{g(\sqrt{2ax})},\sqrt{2ax}\right].$
\end{itemize}
\end{proposition}
This statement implies obviously the following lower bound for the tail of $A_\tau$:
\begin{equation}
\label{lower.bound}
\liminf_{x\to\infty}\frac{\pr(A_\tau>x)}{\overline{F}(\sqrt{2ax})}\ge1.
\end{equation}
Furthermore, using this proposition, one can give an alternative proof of~\eqref{reg.tail} under the assumption of the regular variation of $\overline F$,
which is much simpler than the original one in~\cite{BBP03}.
We first split the event $\lbrace A_\tau>x\rbrace$ into two parts
\begin{align*}
\lbrace A_\tau>x\rbrace=\lbrace A_\tau>x, M_\tau>y\rbrace\cup\lbrace A_\tau>x, M_\tau\leq y\rbrace.
\end{align*}
Clearly, 
\[
\lbrace A_\tau>x, M_\tau\leq y\rbrace\subseteq\lbrace\tau>x/y\rbrace,
\]
and, therefore,
\begin{align}
\label{atausplit}
\pr(A_\tau>x, M_\tau>y)\leq\pr(A_\tau>x)\leq\pr(A_\tau>x, M_\tau>y)+\pr(\tau>x/y).
\end{align}
When $\alpha>1$, according to Theorem I in Doney~\cite{Doney89} or~\cite[Theorem 3.2]{DS13}, 
\begin{equation*}
\pr(\tau>t)\sim\e\tau\bar{F}\left(a t\right)\quad\text{as }t\to\infty.
\end{equation*}
Choosing $y=\varepsilon \sqrt{x}$ and recalling that $\overline{F}$ is regularly varying,
we get
\begin{equation}
\label{tau_tail}
\pr(\tau>x/y)=\pr(\tau>\sqrt{x}/\varepsilon)\sim \varepsilon^\alpha\mathbf{E}\tau\overline{F}(\sqrt{x}).
\end{equation}
It follows from the first statement of Proposition~\ref{prop:joint} that
\begin{equation*}
\mathbf{P}(A_\tau>x,M_\tau>\varepsilon \sqrt{x})\sim\mathbf{E}\tau\overline{F}(\sqrt{2ax}).
\end{equation*}
Plugging this and~\eqref{tau_tail} into~\eqref{atausplit}, we get, as $x\to\infty$,
\begin{equation*}
\mathbf{E}\tau\overline{F}(\sqrt{2ax})(1+o(1))
\le\mathbf{P}(A_\tau>x)\le
\mathbf{E}\tau\overline{F}(\sqrt{2ax})\left(1+\frac{\varepsilon^\alpha}{(2a)^{\alpha/2}}+o(1)\right).
\end{equation*}
Letting $\varepsilon\to0$, we arrive at \eqref{reg.tail}.

The case of semi-exponential distributions is more complicated. 
In particular it seems that in this case there is  a regime  when the asymptotics \eqref{reg.tail} are no longer valid. We will treat this case by using the exponential bounds similar to Section~2.2 in~\cite{PW17} and asymptotics for $\pr(\tau>x)$ from~\cite{DS13} and~\cite{DDS08}. 

First we will introduce a sublclass of subexponential distributions that we will consider. We will assume that $\e[X_1^2]=\sigma^2<\infty$.
Without loss of generality we may assume that $\sigma=1$.
 Let 
    \begin{equation}
      \label{sc1}
    \overline F(x) \sim e^{-g(x)}x^{-2}, \quad  x\to \infty,
    \end{equation}
    where $g(x)$ is an eventually increasing function such that 
    eventually 
    \begin{equation}
      \label{sc2}
    \frac{g(x)}{x^{\gamma_0}}\downarrow 0, \quad x\to\infty,
    \end{equation}
    for some $\gamma_0\in(0,1)$. 
 Due to the asymptotic nature of equivalence in~\eqref{sc1}  without loss of generality we may assume that $g$ is continuously differentiable and  
that~\eqref{sc2} hold for all $x>0$. 
  Clearly, monotonicity in \eqref{sc2} implies 
    \begin{equation}
      \label{sc3}
      g'(x)\le \gamma_0 \frac{g(x)}{x}
    \end{equation}  
    for all sufficiently large $x$. 
    Using the Karamata representation theorem one can show that 
    this class of subexponential distributions  includes regularly varying  
    distributions $\overline F(x)\sim x^{-r}L(x),$ for $r>2$. Also, it is not difficult 
    to show that lognormal distributions and  Weibull distributions 
    ($\overline F(x) \sim e^{-x^\beta},\beta\in(0,1)$) 
    belong to our class of distributions. 
    Previously this class appeared in~\cite{R93} 
    for the analysis of 
    large deviations of sums of subexponential random variables on the whole axis. 

    Now we are able to give rough(logarithmic) asymptotics for $\gamma_0\le 1$. 
\begin{theorem}\label{cor:logarithmic.upper.bound}
    Let $\e[X_1]=-a<0$ and $\mbox{Var}(X_1)<\infty$. 
    Assume that the distribution function $F$ of $X_j$ satisfies~\eqref{sc1} and that~\eqref{sc2} holds with $\gamma_0=1$.
    Then, there exits a constant $C>0$ such that 
    \[
        \pr(A_\tau>x)\le Cx^{1/4} \exp
                \left\{
                    -g(\sqrt{2ax})\sqrt{
                                1-\frac{2Cg(\sqrt{2ax})}{a\sqrt{2ax}}
                            }
                \right\}.
    \]
    Furthermore, for any $\varepsilon>0$ there exist $C>0$ such that, 
    \[
        \liminf_{x\to\infty}\frac{\pr(A_\tau>x)}{\overline F(\sqrt{2ax}+Cx^{1/4+\varepsilon})}\ge \e\tau. 
    \]
    In, particular, if $\gamma_0<1$ then 
   \[
        \lim_{x\to\infty}\frac{\ln \pr(A_\tau>x)}{\ln \overline F(\sqrt{2ax})}=1. 
    \]
\end{theorem}  
To obtain the exact asymptotics we will impose a further assumption
\begin{equation}\label{sc4}
        xg'(x)\to \infty, \quad x\to \infty.
\end{equation}
This assumption implies that
\begin{equation}\label{sc5}
        \frac{g(x)}{\log x}\to \infty.
\end{equation}
In particular, it excludes all regularly varying distributions. 

\begin{theorem}\label{thm:exact.asymptotics}
    Let $\e[X_1]=-a<0$ and $\mathbf{Var}(X_1)<\infty$. Assume that the distribution function $F$ of $X_j$ satisfies~\eqref{sc1}, that~\eqref{sc2} holds with $\gamma_0<1/2$ and that~\eqref{sc4} holds.
    Then, 
    \[
        \pr(A_\tau>x) \sim \e\tau \overline F(\sqrt{2ax}), \quad x\to \infty. 
    \]
\end{theorem}

\subsection{Discussion and organisation of the paper}

In this note we provided exact asymptotics for the case $\gamma_0<1/2$. 
We believe that this restriction is not technical and the asymptotics for  
$\gamma_0\ge 1/2$ is different. This boundary is well-known, for example, 
the same bound appears in the analysis of the exact asymptotics for 
$\pr(\tau>n)$ and $\pr(S_n>an)$, see, correspondingly~\cite{DS13} and~\cite{DDS08}. 

The conjecture in~\cite{KP11} was formulated for the workload process of a single-server queue rather than the area under the random walk excursion. 
However,  one can prove analogous results  for  the L\'evy processes by essentially the same arguments. It is well-known that workload of the M/G/1 queue can be represent as a L\'evy process and thus our results can be transferred to this setting almost immediately.  We believe that the treatment of the workload of the general G/G/1 queue is not that different as well.

The paper is organised as follows. We will start by proving Proposition~\ref{prop:joint} in Section~\ref{sec:area.via.maximum}. 
Then we will derive a useful exponential bound and prove Theorem~\ref{cor:logarithmic.upper.bound} in Section~\ref{sec:logarithmic}. 
Finally we derive exact asymptotics for $\pr(A_\tau>x)$ and thus prove Theorem~\ref{thm:exact.asymptotics} in Section~\ref{sec:asymp.beta.12}.

\section{Proof of Proposition~\ref{prop:joint}}\label{sec:area.via.maximum}
Before giving the proof we will collect some known results that we will need in this and the following Sections. 
We will require the following statement, the first part  of which follows from 
Theorem 2 in Foss, Palmowski and Zachary~\cite{FPZ} (see also~\cite{D2005} for a short proof), and the 
second part from~\cite[Theorem 3.2]{DS13}. 
\begin{proposition}\label{prop:ds13} 
    Let $\e[X_1]=-a$ and either  (a)
     $\overline{F}(x):=\mathbf{P}(X_1>x)=x^{-\alpha}L(x)$ with some $\alpha>1$ 
     or (b)  $\overline{F}(x)\sim x^{-\varkappa} e^{-g(x)}$, 
 where $g(x)$ is a monotone continuously differentiable function satisfying 
 $\frac{g(x)}{x^\beta}\downarrow$ for $\beta\in(0,1/2)$, 
            and $\mathbf{E}|X_1|^\varkappa<\infty$ for some 
            $\varkappa>1/(1-\beta)$  then
            for any fixed $k$, 
\begin{align}
    \label{eq:Sk-MK}
    \pr(M_k>y)&\sim \pr(S_k>y)\sim k\overline F(y), \quad y\to \infty\\
\label{Mtau-k}
\mathbf{P}\left(\max_{n\le\tau\wedge k}S_n>y\right)&\sim \mathbf{E}(\tau\wedge k)\overline{F}(y),\quad  y\to \infty\\
\label{Mtau}
\pr(M_\tau>y)&\sim\e\tau\bar{F}(y),\quad y\to\infty
\end{align}
and
\begin{equation}\label{eq:ds13}
        \pr(\tau>n)\sim \e[\tau]\overline F(an), \quad n\to \infty. 
\end{equation}
\end{proposition}
\begin{proof}
    To prove~\eqref{eq:Sk-MK},~\eqref{Mtau-k} and~\eqref{Mtau}, by Theorem~2 of~\cite{FPZ} it is sufficient to show that (a) or (b) implies that 
    $F\in\mathcal S^*$, that is $\int_0^\infty \overline F(y)<\infty$ and 
    \[
        \int_0^x\overline F(y)\overline F(x-y)dy \sim 2\overline F(x) \int_0^\infty \overline F(y)dy , \quad x\to \infty. 
    \]
    The fact that (a) implies $F\in\mathcal S^*$ is well-known and follows immediately from the dominated confergence theorem, since $ \overline F(x)\sim \overline F(x-y)$ for all fixed $y$ and 
    \[
        \int_0^x \frac{\overline F(y)\overline F(x-y)}{\overline F(x) }dy  
        =2\int_0^{x/2} \frac{\overline F(y)\overline F(x-y)}{\overline F(x) }dy 
    \]
    and $\overline F(x-y)\le C\overline F(x)$ for some $C>0$ when $y\le x/2$.
    
    Now, assume that (b) holds and show that $F\in\mathcal S^*$. 
    Consider now 
    \begin{align*}
        2\int_0^{x/2} \frac{\overline F(y)\overline F(x-y)}{\overline F(x) }dy.
    \end{align*} 
 Uniformly in  $y\in [\ln x ,x/2]$ we have 
 \begin{align*}
\frac{\overline F(y)\overline F(x-y)}{\overline F(x) }
&\le C e^{g(x)-g(x-y)-g(y)} 
=C e^{\int_{x-y}^x g'(t)dt -g(y)} 
\le C e^{\beta \int_{x-y}^x \frac{g(t)}{t}dt -g(x-y)}  \\
&\le C e^{\beta y\frac{g(x-y)}{x-y} -g(x-y)} \le 
Ce^{(\beta-1)g(x-y)} \to 0, \quad x\to\infty,  
 \end{align*}
 and, therefore, 
    \begin{align*}
        2\int_{\ln x }^{x/2} \frac{\overline F(y)\overline F(x-y)}{\overline F(x) }dy\to 0.
    \end{align*} 
    Next for  $y\in [0, \ln x]$, 
\begin{align*}
    1&\le \frac{\overline F(x-y)}{\overline F(x)}\le 
            \frac{\overline F(x-\ln x)}{\overline F(x)}
            \sim e^{g(x)-g(x-\ln x)} = 
        e^{\int_{x-\ln x}^x g'(t)dt}\\ 
     &\le e^{\beta \int_{x-\ln x}^x \frac{g(t)}{t}dt}\le 
     e^{\beta \frac{g(x-\ln x)}{(x-\ln x)^\beta}\int_{x-\ln x}^x t^{\beta-1}dt}
     \le e^{C \frac{g(x-\ln x)}{(x-\ln x)^\beta}\frac{\ln x}{x^{1-\beta}}}\to 1,
\end{align*}    
which implies that $F\in\mathcal S^*$.

    The proof of~\eqref{eq:ds13} is  very similar and can be done by  straightforward verification that~\eqref{sc1} and~\eqref{sc2} imply that conditions of Theorem~3.1  (and hence of Theorem~3.2) of~\cite{DS13} hold.  
\end{proof}

Define 
\[
\sigma_y=\inf\lbrace n<\tau:S_n>y\rbrace.
\]
Then, for every $k\geq 1$,
\begin{align*}
\pr(\sigma_y=k\vert M_\tau>y)&=\frac{\pr(\sigma_y=k)}{\pr(M_\tau>y)}\\
&=\frac{\mathbf{P}\left(\max_{n\le\tau\wedge k}S_n>y\right)-\mathbf{P}\left(\max_{n\le\tau\wedge(k-1)}S_n>y\right)}{\mathbf{P}(M_\tau>y)}.
\end{align*}
It wollows from~\eqref{Mtau-k} and~\eqref{Mtau} that 
\begin{align}
\label{limprob1}
\nonumber
\lim_{y\rightarrow\infty}\pr(\sigma_y=k\vert M_\tau>y)
&=\frac{\mathbf{E}\tau\wedge k-\mathbf{E}\tau\wedge(k-1)}{\e\tau}\\
&=\frac{\mathbf{P}(\tau>k-1)}{\e\tau}=: q_k, \hspace{0,5cm}k\geq 1.
\end{align}
It is clear that
\begin{align*}
\sum_{k=1}^\infty q_k=\frac{1}{\e\tau}\sum_{k=0}^\infty\pr(\tau>k-1)=1.
\end{align*}

For every fixed $N\geq1$ we have
\begin{equation}\label{ThS1}
\begin{split}
\pr(A_\tau&>x, M_\tau>y)\\
&=\sum_{k=1}^N\pr(A_\tau>x, \sigma_y=k, M_\tau>y)+\pr(A_\tau>x, \sigma_y>N, M_\tau>y).
\end{split}
\end{equation}
For the last term on the right hand side we have
\begin{align*}
\pr(A_\tau>x, \sigma_y>N, M_\tau>y)&\leq\pr(\sigma_y>N, M_\tau>y)\\
&=\pr(M_\tau>y)\pr(\sigma_y>N\vert M_\tau>y).
\end{align*}
It follows from~\eqref{limprob1} that  $\pr(\sigma_y>N\vert M_\tau>y)\rightarrow\sum_{j=N+1}^\infty q_j$, as $y\rightarrow\infty$. Then, using \eqref{Mtau}, we get
\begin{align}\label{ataumtau}
\pr(A_\tau>x,\sigma_y>N, M_\tau>y)\leq \varepsilon_N\bar{F}(y),
\end{align}
where $\varepsilon_N\rightarrow 0$ as $N\rightarrow\infty$.

For every fixed $k$ we have
\begin{align*}
\pr(A_\tau>x, \sigma_y=k, M_\tau>y)=\pr(A_\tau>x, \sigma_y=k).
\end{align*}
Since $S_j\in(0,y)$ for all $j<k$, we obtain
\begin{align*}
\pr(A_\tau>x, \sigma_y=k)\leq\pr\left(\sum_{j=k}^{\tau-1}S_j>x-(k-1)y, \sigma_y=k\right)
\end{align*}
and
\begin{align*}
\pr(A_\tau>x, \sigma_y=k)\geq\pr\left(\sum_{j=k}^{\tau-1}S_j>x, \sigma_y=k\right).
\end{align*}
By the Markov property, for  every $z>0$,
\begin{align*}
\pr\left(\sum_{j=k}^{\tau-1}S_j>z, \sigma_y=k\right)=\int_y^\infty\pr(S_k\in dv, \sigma_y=k)\pr(A_\tau>z\vert S_0=v).
\end{align*}

Let $\varkappa\in(1/(1-\beta),2)$ if $\overline{F}$ satisfies the conditions of the part (b)
and let $\varkappa=1$ in the case when $\overline{F}$ is regularly varying. Fix some $\delta>0$
and consider the set 
\[
B_v:=\left\{v-\delta v^{1/\varkappa}\le S_n+na\le v+\delta v^{1/\varkappa}
\mbox{ for all }n\le\frac{v+\delta v^{1/\varkappa}}{a}\right\}.
\]
Since $\e|X_1|^\varkappa<\infty$, it follows from the Marcinkiewicz-Zygmund Law of Large Numbers that
\begin{equation}
\label{refined_lln}
\mathbf{P}(B_v|S_0=v)\to 1\quad\text{as }v\to\infty.
\end{equation}
This implies that, as $y\rightarrow\infty$,
\begin{align*}
\pr&\left(\sum_{j=k}^{\tau-1}S_j>z, \sigma_y=k\right)\\
&=\int_y^\infty\pr\left(S_k\in dv,\sigma_y=k\right)\pr\left(\lbrace A_\tau>z\rbrace\cap B_v\vert S_0=v\right)+o\left(\pr(\sigma_y=k)\right).
\end{align*}
On the event $B_v$ one has
\[
\frac{(v-\delta v^{1/\varkappa})^2}{2a}\leq
A_\tau\leq\frac{(v+\delta v^{1/\varkappa})^2}{2a}.
\]
In other words,
\[
\pr\left(\lbrace A_\tau>z\rbrace\cap B_v\vert S_0=v\right)=\pr(B_v)
\quad\text{if}\quad v-\delta v^{1/\varkappa}\ge\sqrt{2az}
\]
and
\[
\pr\left(\lbrace A_\tau>z\rbrace\cap B_v\vert S_0=v\right)=0
\quad\text{if}\quad v+\delta v^{1/\varkappa}<\sqrt{2az}.
\]
Therefore, for all $v$ large enough,
\begin{align*}
\pr\left(\sum_{j=k}^{\tau-1}S_j>z, \sigma_y=k\right)&\leq\int_{\sqrt{2az}-\delta(2az)^{1/2\varkappa}}^\infty \pr(S_k\in dv, \sigma_y=k)+o(\pr(\sigma_y=k))\\
&=\pr\left(S_{\sigma_y}>\sqrt{2az}-\delta(2az)^{1/2\varkappa}, \sigma_y=k\right)+o(\pr(\sigma_y=k))
\end{align*}
and
\begin{align*}
\pr\left(\sum_{j=k}^{\tau-1}S_j>z, \sigma_y=k\right)&\leq\int_{\sqrt{2az}+2\delta(2az)^{1/2\varkappa}}^\infty \pr(S_k\in dv, \sigma_y=k)\mathbf{P}(B_v)+o(\pr(\sigma_y=k))\\
&=\pr\left(S_{\sigma_y}>\sqrt{2az}+2\delta(2az)^{1/2\varkappa}, \sigma_y=k\right)+o(\pr(\sigma_y=k)).
\end{align*}

\begin{lemma}
For every fixed $k$,
\begin{align*}
\sup_{v>y}\Bigg\vert\frac{\pr(S_k>v,\sigma_y=k)}{\overline{F}(v)}-\pr(\tau>k-1)\Bigg\vert\rightarrow 0\hspace{0,5cm}\text{as } y\rightarrow\infty
\end{align*}
\end{lemma}
\begin{proof}
Fix some $N>0$ and define the events
\[
D_{k,N}=\cup_{j=1}^k\left\lbrace X_j>v+kN, \vert X_l\vert \leq N\hspace{0,3cm} \text{for all } l\neq j, l\le k\right\rbrace.
\]
It is clear that $D_{k,N}\subseteq \lbrace S_k>v\rbrace.$ Therefore,
\begin{align*}
\pr(S_k>v,\sigma_y=k)&=\pr(D_{k,N},\sigma_y=k)+\pr(S_k>v, D^c_{k,N},\sigma_y=k)\\
&=\pr(X_k>v+kN,\vert X_l\vert\leq N, \text{for all } l<k, \sigma_y>k-1)\\
&\hspace{1cm}+\pr(S_k>v, D^c_{k,N},\sigma_y=k).
\end{align*}
For the first term we have $(y>(k-1)N)$
\begin{equation}\label{A}
\begin{split}
\pr&(X_k>v+kN,\vert X_l\vert\leq N, \text{for all } l>k, \sigma_y>k-1)\\
&=\pr(\tau>k-1,\vert X_l\vert\leq N, l<k)\overline{F}(v+kN)\\
&=\pr(\tau>k-1)\overline{F}(v)-\varepsilon_N^{(1)}\overline{F}(v)+o(\overline{F}(v)),\hspace{0,3cm} \text{uniformly in } v>y,
\end{split}
\end{equation}
where
\[
\varepsilon_N^{(1)}:=
\pr(\tau>k-1,\vert X_l\vert> N \text{ for some }l<k)\to0
\quad N\to\infty.
\]

Furthermore,
\begin{equation}\label{B}
\begin{split}
\pr&(S_k>v, D_{k,N}^c,\sigma_y=k)\leq\pr(S_k>v, D_{k,N}^c)=\pr(S_k>v)-\pr(D_{k,N})\\
&=\pr(S_k>v)-k\pr(X_1>v+kN)(\pr(\vert X_1\vert\leq N))^{k-1}\\
&=\varepsilon_N^{(2)}\overline{F}(v)+o(\overline{F}(v)),
\end{split}
\end{equation}
where
\[
\varepsilon_N^{(2)}:=
k\left(1-(\pr(\vert X_1\vert\leq N))^{k-1}\right)\to0,
\quad N\to\infty.
\]

Combining~\eqref{A} and~\eqref{B} and letting $N\rightarrow\infty$ we set the desired relation.
\end{proof}
Since with the previous lemma
\[
\pr(S_{\sigma_y}>v,\sigma_y=k)\sim\bar{F}(v)\pr(\tau>k-1),\hspace{0,5cm} v,y\rightarrow\infty
\]
for $v\geq y$, we infer that
\begin{align*}
&\pr\left(\sum_{j=k}^{\tau-1}S_j>z, \sigma_y=k\right)\leq\bar{F}\left(\sqrt{2az}-\delta(2az)^{1/2\varkappa}\right)(\pr(\tau>k-1)+o(1))\\
&\hspace{4cm}+o(\mathbf{P}(\sigma_y=k))
\end{align*}
and
\begin{align*}
&\pr\left(\sum_{j=k}^{\tau-1}S_j>z, \sigma_y=k\right)\geq\overline{F}\left(\sqrt{2az}+2\delta(2az)^{1/2\varkappa}\right)(\pr(\tau>k-1)+o(1))\\
&\hspace{4cm}+o(\mathbf{P}(\sigma_y=k)).
\end{align*}
Under our assumptions on $\overline{F}$ one has
\[
\lim_{\delta\to0}\lim_{z\to\infty}\frac{\overline{F}\left(\sqrt{2az}+2\delta(2az)^{1/2\varkappa}\right)}
{\overline{F}\left(\sqrt{2az}-\delta(2az)^{1/2\varkappa}\right)}=1.
\]
Therefore,
\begin{align*}
&\pr\left(\sum_{j=k}^{\tau-1}S_j>z, \sigma_y=k\right)
=\overline{F}\left(\sqrt{2az}\right)(\pr(\tau>k-1)+o(1))+o(\mathbf{P}(\sigma_y=k)).
\end{align*}
Consequently,

\begin{align*}
\pr(A_\tau>x,\sigma_y=k)=\bar{F}(\sqrt{2ax})\pr(\tau>k-1)+o\left(\pr(\sigma_y=k)\right).
\end{align*}
Combining~\eqref{Mtau} and~\eqref{limprob1}, one gets
\[
\mathbf{P}(\sigma_y=k)\sim q_k\e\tau\overline{F}(y).
\]
Therefore, 
\begin{align*}
\pr(A_\tau>x,\sigma_y=k)=\bar{F}(\sqrt{2ax})(\pr(\tau>k-1)+o(1))+o(\overline{F}(y)).
\end{align*}
Consequently,
\begin{align}\label{sumpr}
\nonumber
&\sum_{k=1}^N\pr(A_\tau>x, \sigma_y=k, M_\tau>y)\\
&\hspace{1cm}=(\overline{F}(\sqrt{2ax})+o(1))\sum_{k=1}^N\pr(\tau>k-1)
+o(\overline{F}(y)).
\end{align}
Plugging~\eqref{ataumtau} and~\eqref{sumpr} into~\eqref{ThS1} and letting $N\rightarrow\infty$, we obtain
\begin{align*}
\pr(A_\tau>x, M_\tau>y)=(\e\tau+o(1))\bar{F}(\sqrt{2ax})+o(\overline{F}(y)).
\end{align*}
Thus, it remains to show that $\overline{F}(y)=O(\overline{F}(\sqrt{2ax}))$. This is obvious for regularly
varying tails and $y\ge \varepsilon\sqrt{x}$. 

Assume now that $\overline{F}$ satisfies the conditions of part (b).
To simplify notation put $y_*=\sqrt{2ax}-\frac{R\sqrt{2ax}}{g(\sqrt{2ax})}$. 
Then,
\[
 1\le \frac{\overline F(y_*)}{\overline F(\sqrt{2ax})}
 \le (1+o(1))e^{g(\sqrt{2ax})-g\left(y_*\right)}.
\]
Since $\frac{g(x)}{x^\beta}$ is monotone decreasing and $g$ is differentiable 
then clearly 
\[
 g'(x)\le \beta \frac{g(x)}{x}.
\]
Then,
\begin{align*}
 g(\sqrt{2ax})-g\left(y_*\right) 
 &=\int^{\sqrt{2ax}}_{y_*} g'(t) dt 
 \le \beta \int^{\sqrt{2ax}}_{y_*} \frac{g(t)}{t} dt 
 \le \beta \frac{g(y_*)}{(y_*)^\beta}\int^{\sqrt{2ax}}_{y_*} \frac{dt}{t^{1-\beta}} \\ 
 &=\frac{g(y_*)}{(y_*)^\beta}((2ax)^{\beta/2}-(y_*)^\beta)
\le \frac{g(y_*)}{(y_*)^{\beta}}\frac{\beta}{(y_*)^{1-\beta}}
C\frac{\sqrt{2ax}}{g(\sqrt{2ax})}\\
&\le \beta C \frac{\sqrt{2ax}}{y_*}\le (1+o(1))\beta C.
\end{align*}
Therefore, 
\begin{align*}
\overline{F}(y)\leq C\overline{F}(x),\hspace{0,5cm}\forall y\in \left[\sqrt{2ax}-\frac{R\sqrt{2ax}}{g(\sqrt{2ax})},\sqrt{2ax}\right].
\end{align*}

\section{Proof of Theorem~\ref{cor:logarithmic.upper.bound}}
\label{sec:logarithmic}
We start by proving an exponential estimate for the area $A_n$ when random variables $X_j$ are truncated. Let 
\[
    \overline X_n=\max(X_1,\ldots,X_n).
\]
The next result is our main technical tool to investigate trajectories without big jumps.
\begin{lemma}\label{lem:chebyshev}
    Let $\e[X_1]=-a$ and $\sigma^2:=\mathbf{Var}(X_1)<\infty$. 
    Assume that the distribution function $F$ of $X_j$ satisfies~\eqref{sc1} and that~\eqref{sc2} holds with $\gamma_0=1$. 
    Then, there exists a constant $C>0$ such that  
    \[
        \pr(A_n>x,\overline X_n\le y) \le \exp\left\{
                            -\lambda \frac{x}{n}  - \lambda\frac{a n}{2}   +C\lambda^2 n 
                    \right\},
    \]
    where $\lambda = \frac{g(y)}{y}$.
\end{lemma}    
\begin{proof}
We will prove this lemma by using the exponential Chebyshev inequality. 
For that we need to obtain estimates for the moment generating function of 
$A_n$. First,
\[
    \e\left[ e^{\frac{\lambda}{n}A_n};\overline X_n\le y\right]=
    \e \left[e^{\frac{\lambda}{n}\sum_1^n(n-j+1)X_j};\overline X_n\le y\right]
    =\prod_{j=1}^n\varphi_y\left(\lambda_{n,j}\right),
\]
where 
\[
    \varphi_y(t) := \e[e^{t X_j};X_j\le y]
\]
and 
\[
    \lambda_{n,j} := \lambda\frac{(n-j+1)}{n}.
\]
Then,
\begin{align*}
    \varphi_y(\lambda_{n,j})&= \e[e^{\lambda_{n,j}X_j}; X_j\le 1/\lambda_{n,j}]+
    \e[e^{\lambda_{n,j}X_j};1/\lambda_{n,j} < X_j\le y]\\
                            &=: E_1+E_2.
\end{align*}    
Using the elementary bound $e^x\le 1+x+x^2$ for $x\le 1$ we obtain, 
\begin{align*}
    E_1\le 1+\lambda_{n,j}\e[X_j]
    +\lambda_{n,j}^2 \e[X_j^2]
    =1-a\lambda_{n,j}+(a^2+\sigma^2)\lambda_{n,j}^2. 
\end{align*}    
Next, using the integration by parts and the assumption~\eqref{sc1}, 
\begin{align*}
    E_2&=\int_{1/\lambda_{n,j}}^{y} e^{\lambda_{n,j}t}dF(t) 
            =-\overline F(t)e^{\lambda_{n,j}t}\biggl |_{t=1/\lambda_{n,j}}^{t=y}
            +\lambda_{n,j}\int_{1/\lambda_{n,j}}^{y} e^{\lambda_{n,j}t}\overline F(t) dt\\     
       &\le e \overline F(1/\lambda_{n,j})+C \lambda_{n,j}\int_{1/\lambda_{n,j}}^{y} e^{\lambda_{n,j}t-g(t)}t^{-2}dt. 
\end{align*}    
Now note that for $t\le y$, 
\[
    \lambda_{n,j} t - g(t) = t \left(\lambda_{n,j}-\frac{g(t)}{t}\right)
    \le t \left(\lambda_{n,j}-\frac{g(y)}{y}\right),
\]
due to the condition~\eqref{sc2}. Then, 
\[
    \lambda_{n,j}-\frac{g(y)}{y} \le 
    \lambda - \frac{g(y)}{y}=0
\] 
and, therefore, 
\[
    E_2\le e \overline F(1/\lambda_{n,j})+C \lambda_{n,j}\int_{1/\lambda_{n,j}}^{y} t^{-2}dt \le (C+e)\lambda_{n,j}^2, 
\]
where we also used the Chebyshev inequality. 
As a result, for some constant $C$, 
\[
    \varphi_y(t)= E_1+E_2\le 1-a\lambda_{n,j} +C\lambda_{n,j}^2.
\]
Consequently, 
\begin{align*}
    \e\left[e^{\frac{\lambda}{n}A_n};\overline X_n\le y\right]&\le \prod_{j=1}^n \left(1-a\lambda_{n,j} +C\lambda_{n,j}^2\right) \\
    &= 
    \exp\left\{
        \sum_{j=1}^n\ln \left(1-a\lambda_{n,j} +C\lambda_{n,j}^2\right)  
    \right\}\\
                               &\le 
                               \exp\left\{
                                    \sum_{j=1}^n\left(-a\lambda_{n,j} +C\lambda_{n,j}^2\right)  
                                \right\}\\
                               &=\exp\left\{
                                        \sum_{j=1}^n\left(-a\lambda\frac{n-j+1}{n} +C\left(\lambda\frac{n-j+1}{n} \right)^2\right)  
                                \right\}\\
                               &\le \exp\left\{
                                   -\frac{a \lambda}{2} n +C\lambda^2 n 
                               \right\}.
\end{align*}    
Finally,
\[
    \pr(A_n>x,\overline X_n\le y)\le e^{-\lambda\frac{x}{n}}\e\left[e^{\frac{\lambda}{n}A_n};\overline X_n\le y\right] 
    \le 
                \exp\left\{
                            -\lambda \frac{x}{n}  -\frac{a \lambda}{2}   n +C\lambda^2 n 
                    \right\}.
\]

\end{proof}
We can now obtain a rough upper bound using the exponential bound in Lemma~\ref{lem:chebyshev}.
\begin{lemma}\label{lem:logarithmic.upper.bound}
    Let $\e[X_1]=-a<0$ and $\mathbf{Var}(X_1)<\infty$. 
    Assume that the distribution function $F$ of $X_j$ satisfies~\eqref{sc1} and that~\eqref{sc2} holds with $\gamma_0=1$.
    Then, there exists a constant $C>0$ such that 
    \[
        \pr(A_\tau>x)\le Cx^{1/4} \exp
                \left\{
                    -g(\sqrt{2ax})\sqrt{
                                1-\frac{2Cg(\sqrt{2ax})}{a\sqrt{2ax}}
                            }
                \right\}
    \]
\end{lemma} 

\begin{proof}
 Clearly, 
 \[
     \pr(A_\tau>x)\le \pr(A_\tau>x, \overline X_\tau\le \sqrt{2ax})
     +\pr(A_\tau>x, \overline X_\tau>\sqrt{2ax})=:P_1+P_2. 
 \]
 First, using Lemma~\ref{lem:chebyshev} with $y=\sqrt{2ax}$ we obtain, 
 \begin{align*}
     P_1&\le \sum_{n=0}^\infty\mathbf{P}(A_n\geq x, \overline X_n \le \sqrt{2ax},\tau=n+1)\\ 
        &\le \sum_{n=1}^\infty
                    \exp\left\{
                        -\lambda \frac{x}{n}  -\frac{a \lambda}{2} n +C\lambda^2 n 
                    \right\}
                    =
                    \sum_{n=1}^\infty
                    \exp\left\{
                            -\lambda \frac{x}{n}  - \lambda I n 
                    \right\},
 \end{align*}    
 where $\lambda = \frac{g(\sqrt {2ax} )}{\sqrt {2ax}}$ and $I=\frac{a}{2}-C\lambda. $
With formula (25) at page 146 of Bateman~\cite{BTIT} we have,
\begin{align*}
\sum_{n=1}^\infty\exp\left\lbrace-\lambda\frac{x}{n}-\lambda I n\right\rbrace&\leq \int_0^\infty\exp\left\lbrace-\lambda \frac{x}{y}-\lambda I(y+1)\right\rbrace dy\\
&=e^{-\lambda I}\sqrt{\frac{4x}{I}}K_1(2\lambda\sqrt{Ix}).
\end{align*}
Now using the asymptotics for the modified Bessel function
\[
K_1(z)\sim \sqrt{\frac{\pi}{2z}}e^{-z}
\]
we obtain
\begin{align*}
\sum_{n=1}^\infty\exp\left\lbrace-\lambda\frac{x}{n}-\lambda I n\right\rbrace\leq Cx^{1/4}\exp\{-2\lambda\sqrt{Ix}\}.
\end{align*}
Therefore, 
\begin{align}\label{eq:p1}
    P_1&\le  Cx^{1/4}\exp\{-2\lambda\sqrt{Ix}\}\\
    \nonumber
    &\le Cx^{1/4} \exp
                \left\{
                    -g(\sqrt{2ax})\sqrt{
                                1-\frac{2Cg(\sqrt{2ax})}{a\sqrt{2ax}}
                            }
                \right\}.
\end{align}    
Next,
\begin{align*}
    P_2&\le \sum_{n=0}^\infty\mathbf{P}(A_\tau\geq x, M_n \le \sqrt{2ax}, X_{n+1}>\sqrt{2ax} ,\tau>n)\\
        &\le \sum_{n=0}^\infty\mathbf{P}( X_{n+1}>\sqrt{2ax}) \pr(\tau>n) \le 
        \e[\tau] \overline F(\sqrt {2a x}) = o(P_1).
\end{align*}    
Then, the claim follows. 
\end{proof}    
Now we will give a lower bound.
\begin{lemma}\label{lem:lower.bound}
 Let $\e[X_1]=-a<0$ and $\mathbf{Var}(X_1)<\infty$. 
 Then, for any $\varepsilon>0$ there exists $C>0$ such that, 
 \[
  \liminf_{x\to\infty}\frac{\pr(A_\tau>x)}{\overline F(\sqrt{2ax}+Cx^{1/4+\varepsilon})}\ge \e\tau. 
 \]
\end{lemma}
\begin{proof}
 Fix $N\ge 1$. 
 Put $y^+ =\sqrt{2ax}+Cx^{1/2-\varepsilon},$ 
 where $C$ will picked later. 
 Since $\e[X_1^2]<\infty$, by the Strong Law of Large Numbers,  
 \[
  \frac{S_{l}+al}{l^{1/2+\varepsilon}}\to 0, \quad l\to \infty    \mbox{ a.s.}
 \]
 Hence, for any $\delta>0$ we can pick $R>0$ such that 
\[
 \pr\left( \min_{l\le \sqrt{2x/a}} (S_{l}+al +R + l^{1/2+\varepsilon})>0\right)>(1-\delta)
\]
Now note that there exists a sufficiently large $C$ such that,
for every $k\le N$,
\[
 \left\{\min_{l\le \sqrt{2x/a}} (S_{k+l}-S_k+al +R + l^{1/2+\varepsilon})>0,\tau>k, S_k>y^+\right\}\subset 
 \{A_\tau>x\}
 \]
Hence,
\begin{align*}
  &\pr(A_{\tau}>x) \ge  \sum_{k=0}^N\pr(A_\tau>x, \overline X_{k-1}\le y^+, X_k>y^+, \tau>k)\\ 
  &\ge 
  \sum_{k=0}^N\pr\left(\overline X_{k-1}\le y^+,\tau>k-1, X_k>y^+, \min_{l\le \sqrt{2x/a}} (S_{l+k}-S_k +R + j^{1/2+\varepsilon})>0\right)\\
  &\ge (1-\delta)\sum_{k=0}^N \pr\left(\overline X_{k-1}\le y^+,\tau>k-1\right)\overline F(y^+).
 \end{align*}
 For every fixed $k$ we have
 $$
 \pr\left(\overline X_{k-1}\le y^+,\tau>k-1\right)
 \to \pr\left(\tau>k-1\right),\quad x\to\infty.
 $$
 Furthermore, $\sum_{k=0}^N \pr(\tau>k)\to\e\tau$ as
 $N\to\infty$.  Therefore, we can pick sufficiently large $N$  
 such that 
 \[
  \liminf_{x\to\infty}\sum_{k=0}^N \pr\left(\overline X_{k-1}\le y^+,\tau>k-1\right)\ge (1-\delta)\e\tau.
 \]
 Then, for all $x$ sufficiently large,
 \[
  \pr(A_{\tau}>x)\ge (1-\delta)^2\e\tau\overline F(y^+).
 \]
 As $\delta>0$ is arbitratily small we arrive at the conclusion.
\end{proof}

{\it Completion of the proof of Theorem~\ref{cor:logarithmic.upper.bound}.}
 The upper bound follows from Lemma~\ref{lem:logarithmic.upper.bound}. 
 The lower bound follows from Lemma~\ref{lem:lower.bound}. The rough asymptotics follows immediately from the lower and upper bounds and  from the observation that 
 \begin{equation}
 \label{observation}
  \sup_{|y|\le x\rho(x)}\left|\frac{\log \overline F(x)}{\log \overline F(x+y)}-1\right|\to0,
 \end{equation}
 where $\rho(x)\to0$.
 
 To prove \eqref{observation} we note that by \eqref{sc2} and \eqref{sc3}
 \begin{align}\label{eq:insensitivity}
  g(x+y)-g(x)&=\int_x^{x+y} g'(t) dt \le 
  \gamma_0 \int_x^{x+y} \frac{g(t)}{t}  dt
  \le \gamma_0 \frac{g(x)}{x^{\gamma_0}}\int_x^{x+y} \frac{1}{t^{1-\gamma_0}}  dt\\
  \nonumber 
  &\le \gamma_0\frac{g(x)}{x^{\gamma_0}}\frac{y}{x^{1-\gamma_0}} 
  =\gamma_0g(x)\frac{y}{x},\quad y>0. 
 \end{align}
 This implies that, as $x\to\infty$, 
\begin{equation}
\label{g-ratio}
 \sup_{|y|\le x\rho(x)}\left|\frac{g(x+y)}{g(x)}-1\right|\to0.
\end{equation}
Recalling that
$$
\log\overline{F}(x)\sim-g(x)-2\log x,
$$
one obtains easily \eqref{observation}.
\section{Proof of Theorem~\ref{thm:exact.asymptotics}}\label{sec:asymp.beta.12}
Set 
\[
    h(x):= \frac{\sqrt{2ax}}{g(\sqrt{2ax})}
\]
and 
\begin{equation}\label{eq:y-}
    y=\sqrt{2ax} - Ch(x) \log x ,
\end{equation}   
where $C>\frac{5/4}{1-\gamma_0}$.  
First we will split the probability $\pr(A_{\tau}>x)$ as follows
\begin{align*}
 \pr(A_{\tau}>x)&=\pr(A_{\tau}>x, \overline X_\tau\le y)
                +\pr\left(A_{\tau}>x, \overline X_\tau> \sqrt{2ax}-\frac{1}{\log x}h(x)\right)\\
                &+\pr\left(A_{\tau}>x, \overline X_\tau\in \left[y, \sqrt{2ax}-\frac{1}{\log x}h(x))\right]\right)=:P_1+P_2+P_3.
\end{align*}
The first term will be estimated using the exponential bound proved in Lemma~\ref{lem:chebyshev}. 
\begin{lemma}\label{lem:p1}
Let $\e[X_1]=-a$ and $\mathbf{Var}(X_1)<\infty$. 
Assume that~\eqref{sc1} and~\eqref{sc2} hold for some $\gamma_0<1/2$ together with~\eqref{sc4}. Then, 
    \[
        P_1 = o(\overline F(\sqrt{2ax})).
    \]
\end{lemma}
\begin{proof}
According to~\eqref{eq:p1},
 \begin{align*}
    P_1&\le  Cx^{1/4}\exp\{-2\lambda\sqrt{Ix}\},
 \end{align*}
 where $I=\frac{a}{2}-C\lambda$ and $\lambda = g(y)/y$.
 Since~\eqref{sc2} holds for some $\gamma_0<1/2$, 
 $g^2(y)/y\to 0$ and hence 
 \[
  P_1 \le Cx^{1/4}\exp\left\{-\frac{g(y)}{y}\sqrt{2ax}\right\}.
 \]
 Then,
 \[
  \frac{P_1}{\overline F(\sqrt{2ax})}\le C x^{5/4}
  \exp\left\{g(\sqrt{2ax})-\frac{g(y)}{y}\sqrt{2ax}\right\}.
 \]
 To finish the proof it is sufficient to show that 
 \begin{equation}\label{eq:p1.intermediate}
  g(\sqrt{2ax})-\frac{g(y)}{y}\sqrt{2ax} + \frac{5}{4}\log x \to -\infty,\quad  x\to \infty. 
 \end{equation}
We first note that 
\begin{align*}
 d(x)&:=g(\sqrt{2ax})-\frac{g(y)}{y}\sqrt{2ax}  = 
 g(\sqrt{2ax}) - \frac{g(y)}{1-C\frac{\log x}{g(\sqrt{2ax})}}\\
     &= g(\sqrt{2ax}) - g(y) + (C+o(1)) \log x \frac{g(y)}{g(\sqrt{2ax})}.
\end{align*}
Using~\eqref{sc3} and~\eqref{sc2} one can see that 
\begin{align}
 \label{eq:bound.g}   
 g(\sqrt{2ax}) - g(y) &= \int_y^{\sqrt{2ax}} g'(z) dz 
 \le \gamma_0\int_y^{\sqrt{2ax}} \frac{g(z)}{z} dz 
 \le \gamma_0 \frac{g(y)}{y} (\sqrt{2ax}-y)\\
 &=\gamma_0 C \frac{g(y)}{y}\log x \frac{\sqrt{2ax}}{g(\sqrt{2ax})}.
 \nonumber
\end{align}
Hence, 
\[
 d(x)\le \left(\gamma_0\frac{\sqrt{2ax}}{y}-1 \right) (C+o(1)) \frac{g(y)}{g(\sqrt{2ax})}\log x. 
\]
According to~\eqref{g-ratio}, $g(y)\sim g(\sqrt{2ax})$.
Therefore,~\eqref{eq:p1.intermediate} is valid for any $C$ satisfying $C(\gamma_0-1)+\frac{5}{4}<0$. 

\end{proof}
Next lemma gives the term with the main contribution. 
\begin{lemma}\label{lem:p2}
 Under the assumptions of Lemma~\ref{lem:p1} we have the following estimate 
 \[
  P_2\le (1+o(1))\overline F(\sqrt{2ax}), \quad x\to \infty. 
 \]
\end{lemma}
\begin{proof}
Put 
\[
 y^*=\sqrt{2ax}-\frac{h(x)}{\log x}. 
\]

 By the total probability formula, 
 \begin{align*}
  P_2&\le \sum_{n=0}^\infty\mathbf{P}(A_\tau\geq x, \overline{X}_n \le y^*, X_{n+1}>y^* ,\tau>n)\\
        &\le \sum_{n=0}^\infty\mathbf{P}( X_{n+1}>y^*) \pr(\tau>n) 
        = \e[\tau] \overline F(y^*).
 \end{align*}
Now note that by~\eqref{eq:bound.g} and~\eqref{g-ratio}
\begin{align*}
 \frac{\overline F(y^*)}{\overline F(\sqrt{2ax})}
    &\le (1+o(1)) e^{g(\sqrt{2ax})-g(y^*)}
    \le (1+o(1)) e^{\frac{\gamma_0 g(y^*)}{y^*}(\sqrt{2ax}-y^*)}\\ 
    &\le (1+o(1)) e^{\frac{\gamma_0 g(y^*)}{y^*}\frac{1}{\log x}\frac{\sqrt{2ax}}{g(\sqrt{2ax})}} = 1+o(1). 
\end{align*}
Then the statement immediately follows. 

\end{proof}

We will proceed to the analysis of  $P_3$. 
Fix some $\delta>0$ and set
\[
 z=\frac{1}{a}\left(\sqrt{2ax} + \delta\sqrt{x} \right).
\]
We will split $P_3$ further as follows, 
\begin{align*}
 P_3\le P_{31}+P_{32}+P_{33} &:= 
 \pr\left(A_{\tau}>x, \overline X_\tau\in \left[y, \sqrt{2ax}-R(x)h(x)\right];J_1;\tau\le 
 z\right )\\
 &+
 \pr\left(A_{\tau}>x, \overline X_\tau\in \left[y, \sqrt{2ax}-R(x)h(x)\right];J_{\ge 2}, \tau\le z\right)\\
 &+\pr(\tau>z),
\end{align*}
where 
\[
 J_1 =\left\{
  \mbox{there exists $k\in (1, \tau)$ such that } X_k>y \mbox{ and } 
  \max_{1\le i\le \tau, i\neq k} X_i \le y
 \right\}
\]
and, correspondingly, 
\[
 J_{\ge 2} =\left\{
  \mbox{there exist $k, l\in (1, \tau)$ such that } X_k>y \mbox{ and  } X_l>y
 \right\}
\]

We will start with easier terms $P_{32}$ and $P_{33}$.
To deal with these terms we  will use Proposition~\ref{prop:ds13}. 
One can see then 
\begin{lemma}\label{lem:p33}
Let the assumptions~\eqref{sc1},\eqref{sc2} and \eqref{sc4} hold for $\gamma_0<1/2$. Then,
\[
 P_{33} = o(\overline F(\sqrt{2ax})), \quad x\to\infty.
\]
\begin{proof}
 We have, by Proposition~\ref{prop:ds13}, 
 \[
  P_{33}\le \pr(\tau>z)\le (\mathbf{E}\tau+o(1)) 
  \overline F(az)  = 
  O\left(\overline F(\sqrt{2ax}+\delta \sqrt{x})\right).
 \]
 Therefore,
 \begin{align*}
  \frac{P_{33}}{\overline F(\sqrt{2ax})}
  &\le C
  e^{g(\sqrt{2ax})-g(\sqrt{2ax}+\delta \sqrt{x})}.
 \end{align*}
By the mean value theorem and by the assumption~\eqref{sc4},
\[
g(cx)-g(x)\to\infty,\quad x\to\infty
\]
for every $c>1$. This completes the proof.
\end{proof}

\end{lemma}

\begin{lemma}\label{lem:p32}
 Let the conditions of Lemma~\ref{lem:p2} hold. 
Then, 
\begin{equation}\label{eq:p32}
    P_{32} = o(\overline F(\sqrt{2ax})). 
\end{equation}
\end{lemma}
\begin{proof}
 We can use the formula of total probability to write
 \[
  P_{32} \le \sum_{k=1}^z \pr(\tau>k, J_{\ge 2})
  \le \sum_{k=1}^z \frac{k^2}{2}\overline F(y)^2.
 \]
Then,
\[
  \frac{P_{32}}{\overline F(\sqrt{2ax})}
  \le C x^{3/2}\frac{\overline F(y)^2}{\overline F(\sqrt{2ax})}
  \le C x^{1/2} e^{g(\sqrt{2ax})-2g(y)}.
\]
Using now~\eqref{eq:bound.g} one can see that 
\begin{align*}
 \frac{P_{32}}{\overline F(\sqrt{2ax})} 
 \le C x^{1/2} e^{C\ln x -g(y)}\to 0,
\end{align*}
in view of~\eqref{sc5}.
\end{proof}

We are left to analyse $P_{31}$. For that introduce 
\[
 \mu(y):=\min\{n\ge 1: X_k>y\}.
\]
Now we will complete the proof with the following Lemma. 
\begin{lemma}\label{lem:q1}
 Let the assumptions~\eqref{sc1},\eqref{sc2} and~\eqref{sc4} hold for $\gamma_0<1/2$. Then,
\[
 P_{31} = o(\overline F(\sqrt{2ax})), \quad x\to\infty.
\]
\end{lemma}
\begin{proof}
First represent event $J_1=J_{11}\cup J_{12}$, where
\begin{align*}
 J_{11}&:=
 \{
    \mbox{$X_k>y$ for exactly one $k\in(0,\tau)$ and 
    $X_i\le x^\varepsilon$ for all other 
        $i<\tau$ }
    \}\\
J_{12}&:=
 \{
    \mbox{$X_k>y$ for exactly one $k\in(0,\tau)$ and 
          $X_i>x^\varepsilon $ for some
        $i\neq k, i<\tau$}
    \}.    
\end{align*}
Then, 
\begin{align*}
 Q_{2}&:=\pr\left(A_{\tau}>x, \overline X_\tau\in \left[y, \sqrt{2ax}-\frac{1}{\log x}h(x)\right];J_{12},
 \tau\le 
 z\right )\\ 
 &\le \sum_{j=1}^z \pr(\tau=j, J_{12})
 \le \sum_{j=1}^z \frac{j^2}{2}\overline F(y)\overline F(x^\varepsilon) 
 \le z^3 \overline F(y)\overline F(x^\varepsilon).
  \end{align*}

Then, 
\begin{align*}
 \frac{Q_{2}}{\overline F(\sqrt{2ax})}\le 
 Cx^{3/2+2\varepsilon} e^{g(\sqrt{2ax})-g(y)-g(x^{\varepsilon}))}
\end{align*}
By~\eqref{eq:bound.g}, 
\[
 g(\sqrt{2ax})-g(y)\le C\ln x.
\]
Then, in view of the relation~\eqref{sc5} we have 
\[
 g(\sqrt{2ax})-g(y)-g(x^{\varepsilon}))\le -4\ln x, 
\]
which implies that $Q_{2} = o(\overline F(\sqrt{2ax}))$.

To estimate 
\[
Q_{1}:=\pr\left(A_{\tau}>x, \overline X_\tau\in \left[y, \sqrt{2ax}-\frac{1}{\log x}h(x)\right];J_{11},
 \tau\le 
 z\right )
\]
we make use of the exponential bound given in Lemma~\ref{lem:chebyshev}. Put  
putting 
\[
    x^+(k)=x-k\left(\sqrt{2ax}-\frac{h(x)}{\log x}\right).
\]
Then, we have, 
\begin{align*}
    Q_{1}&= 
    \sum_{k=0}^{z-1} \sum_{j=1}^k \pr\left(A_k>x, \max_{i\neq j, i\le k} X_i\le x^\varepsilon, X_j\in\left[y, \sqrt{2ax}-\frac{h(x)}{\log x}\right] , \tau=k+1\right)\\
 &\le 
 \sum_{k=1}^z (k+1)\pr(A_k>x^+(k), \overline X_k\le x^\varepsilon)\overline F(y)\\
 &\le Cx^{1/2}\overline F(y)
 \sum_{k=1}^z 
 \exp\left\{
                        -\lambda \frac{x^+(k)}{k}  -\frac{a \lambda}{2} k +C\lambda^2 k 
                    \right\},
\end{align*}
where $\lambda = \frac{g(x^\varepsilon)}{x^\varepsilon}$.
Now note that 
\[
 -\lambda \frac{x^+(k)}{k}  -\frac{a \lambda}{2} k 
 =-\lambda\left(-\sqrt{2ax}+\frac{h(x)}{\log x} +\frac{x}{k}+\frac{ak}{2}\right).
\]
Since 
\[
 \frac{x}{k}+\frac{ak}{2}\ge \sqrt{2ax},\quad k\ge1, 
\]
we obtain, 
\[
 -\lambda \frac{x^+(k)}{k}  -\frac{a \lambda}{2} k 
 \le -\lambda \frac{h(x)}{\log x},\quad k\ge1.
\]
Thus, 
\[
 Q_{1} \le Cx e^{-\lambda h(x)/\log x+\lambda^2 z}
 \overline{F}(y).
\]
Next, we can pick $\varepsilon = \frac{1}{4(1-\gamma_0)} $ to achieve  
\begin{align*}
 \lambda^2 z &\le C \left(\frac{g(x^\varepsilon)}{x^\varepsilon}\right)^2x^{1/2}
 =C 
 \left(\frac{g(x^\varepsilon)}{x^{\varepsilon(1-1/(4\varepsilon))}}\right)^2
 =C 
 \left(\frac{g(x^\varepsilon)}{x^{\gamma_0\varepsilon}}\right)^2\\
 &<C\sup_t\left(\frac{g(t)}{t^{\gamma_0}}\right)^2<\infty,
\end{align*}
by the condition~\eqref{sc2}. Note that since $\gamma_0<1/2$, the 
picked $\varepsilon<1/2$ as well. 
Then,
\begin{align*}
 \frac{Q_{1}}{\overline F(\sqrt{2ax})}\le 
 Cx^{2} e^{g(\sqrt{2ax})-g(y)-\lambda h(x)/\log x},
\end{align*}
and using~\eqref{eq:bound.g}, 
\begin{align*}
 \frac{Q_{1}}{\overline F(\sqrt{2ax})}\le 
 Cx^{C} e^{-\lambda h(x)/\log x}.
\end{align*}
Finally noting that 
\begin{align*}
 \lambda h(x) = \frac{g(x^\varepsilon)}{x^{\varepsilon}} \frac{\sqrt{2ax}}{g(\sqrt{2ax})}
\end{align*}
is decreasing polynomially we obtain required convergence to $0$. 
The polynomial decay can be immediately seen for  $g(x)=x^{\gamma_0}$. 
 However, a proper proof goes as follows,  
\begin{align*}
 g(C\sqrt{x}) &= g(x^\varepsilon)+\int_{x^\varepsilon}^{C\sqrt{x}} g'(t)dt
 \le g(x^\varepsilon)+\gamma_0\int_{x^\varepsilon}^{C\sqrt{x}} \frac{g(t)}{t}dt\\
 &\le g(x^\varepsilon)+
 \gamma_0\int_{x^\varepsilon}^{C\sqrt{x}} \frac{g(t)}{t^{\gamma_0}} t^{\gamma_0-1}dt
 \le g(x^\varepsilon)+\frac{g(x^\varepsilon)}{x^{\varepsilon\gamma_0}}\int_{x^\varepsilon}^{C\sqrt{x}} t^{\gamma_0-1}dt\\
 &\le g(x^\varepsilon)+C\frac{g(x^\varepsilon)}{x^{\varepsilon \gamma_0}} x^{\gamma_0/2}
 \le C g(x^\varepsilon)x^{\gamma_0(1/2-\varepsilon)}
\end{align*}
Therefore,
\[
\lambda h(x)\ge x^{1/2-\varepsilon} x^{-\gamma_0(1/2-\varepsilon)}
\]
\end{proof}

{\it Completion of the proof of Theorem~\ref{thm:exact.asymptotics}}
Combination of the preceding Lemmas give us the upper bound. 
The lower bound has been shown in~\eqref{lower.bound} under even weaker conditions.


\end{document}